\documentclass[11pt]{article}
\usepackage{amsfonts,amsthm}

\newtheorem{remark}{Remark}

\newtheorem{theorem}{Theorem}
\newtheorem{corollary}{Corollary}
\newtheorem{lemma}{Lemma}

\newcommand{\C}{\mathbb{C}}

\newcommand{\cA}{\mathcal{A}}

\newcommand{\cL}{\mathcal{L}}
\newcommand{\cQ}{\mathcal{Q}}

\begin{document}

\title{Unified spectral bounds on the chromatic number}

\author{
Clive Elphick\thanks{\texttt{clive.elphick@gmail.com}}
\and
Pawel Wocjan\thanks{Department of Electrical Engineering and Computer Science, University of Central Florida, Orlando, USA; \texttt{wocjan@eecs.ucf.edu}}
}

\date{October 29, 2014}

\maketitle


\begin{abstract}
One of the best known results in spectral graph theory is the following lower bound on the chromatic number due to Alan Hoffman, where $\mu_1$ and $\mu_n$ are respectively the maximum and minimum eigenvalues of the adjacency matrix: $\chi \ge 1 + \mu_1/-\mu_n$. We recently generalised this bound to include all eigenvalues of the adjacency matrix.
 
In this paper, we further generalize these results to include all eigenvalues of the adjacency, Laplacian and signless Laplacian matrices. The various known bounds are also unified by considering the normalized adjacency matrix, and examples are cited for which the new bounds outperform known bounds.
\end{abstract}


\section{Introduction}

We introduce some notation to state the lower bounds on the chromatic number.
Let $G$ be a graph with $n$ vertices, $m$ edges, chromatic number $\chi$ and adjacency matrix $A$.  Let $D$ be the diagonal matrix of vertex degrees.  Let $L=D-A$ denote the Laplacian of $G$ and $Q=D+A$ denote the signless Laplacian of $G$.
The eigenvalues of $A$ are denoted by $\mu_1\ge\ldots\ge \mu_n$; of $L$ by $\theta_1 \ge \ldots \ge \theta_n =0$; of $Q$ by $\delta_1\ge\ldots\ge\delta_n\ge 0$.  It is known that for all graphs $\delta_i \ge 2 \mu_i$ holds for $i=1,\ldots,n$.

\begin{theorem}[Lower bounds on the chromatic number]\label{thm:bounds}
The chromatic number is bounded from below by:
\begin{eqnarray}
\chi & \ge & 1 + \frac{\mu_1}{-\mu_n} \label{eq:Hoffman} \\
\chi & \ge & 1 + \frac{\mu_1}{\theta_1 - \mu_1} \label{eq:Niki} \\
\chi & \ge & 1 + \frac{\mu_1}{\mu_1 - \delta_1 + \theta_1} \label{eq:Kolo1} \\
\chi & \ge & 1 + \frac{\mu_1}{\mu_1 - \delta_n + \theta_n} \label{eq:Kolo2}
\end{eqnarray}
The bound (\ref{eq:Hoffman}) was proved by Hoffman \cite{Hoffman70} in 1970, the bound (\ref{eq:Niki}) by Nikiforov \cite{Nikiforov07} in 2007, and the bounds (\ref{eq:Kolo1}) and (\ref{eq:Kolo2}) by Kolotilina \cite{Kolotilina} in 2010.  Observe that $\theta_n=0$ is included in (\ref{eq:Kolo2}) on purpose because the generalization of this bound makes use of the eigenvalues of $L$.
\end{theorem}

The purpose of this paper is to prove the following multi-eigenvalue generalization of the above lower bounds.

\begin{theorem}[Generalized lower bounds on the chromatic number]\label{thm:genBounds}
The chromatic number is bounded from below by:
\begin{eqnarray}
\chi & \ge & 1 + \frac{\sum_{i=1}^m \mu_i}{-\sum_{i=1}^m \mu_{n+1-i}} \label{eq:genHof} \\
\chi & \ge & 1 + \frac{\sum_{i=1}^m \mu_i}{\sum_{i=1}^m (\theta_i - \mu_i)} \label{eq:genNiki} \\
\chi & \ge & 1 + \frac{\sum_{i=1}^m \mu_i}{\sum_{i=1}^m (\mu_i - \delta_i + \theta_i)} \label{eq:genKolo1} \\
\chi & \ge & 1 + \frac{\sum_{i=1}^m \mu_i}{\sum_{i=1}^m (\mu_i - \delta_{n+1-i} + \theta_{n+1-i})} \label{eq:genKolo2}
\end{eqnarray}
for all $m=1,\ldots,n$.
Bound (\ref{eq:genHof}) was proved by Wocjan and Elphick \cite{WE} in 2013. The other bounds are new.
\end{theorem}

\begin{remark}
In 2011 Lima, Oliveira, Abreu and Nikiforov \cite{LOAN} proved that 
\begin{equation}\label{eq:LOAN}
\chi \ge 1 + \frac{2m}{2m-n \delta_n}.
\end{equation}
A new proof of this result is provided, based on the method of converting the adjacency matrix into the zero matrix.  Observe that since $\mu_1 \ge 2m/n$, bound (\ref{eq:Kolo2})  follows immediately from this result.  
\end{remark}

\section{Proof of Theorem~\ref{thm:bounds}}
To put the generalized multi-eigenvalue lower bounds on the chromatic number in Theorem~\ref{thm:genBounds} and our proof into context, we outline the proof of Theorem~\ref{thm:bounds}.

Let $V=\{1,\ldots,n\}$.  Let $A\in\C^{n\times n}$ be a Hermitian matrix having zeros on the diagonal.  We say that $A$ can be colored with $c$ colors if there exists a partition of $V$ into disjoint subsets $V_1,\ldots,V_c$ such that for each $s=1,\ldots,c$ $a_{k\ell}=0$ for all $k,\ell\in V_s$.

In the special case when $A$ is the adjacency matrix of a graph, this corresponds to the usual graph coloring.  A graph can be colored with $c$ colors if it is possible to assign at most $c$ different colors to its vertices such that any two adjacent vertices receive different colors. The chromatic number $\chi$ is the minimum number of colors required to color the graph.

In 2007, Nikiforov proved the following result \cite{Nikiforov07}:

\begin{lemma}\label{lem:Niki}
Let $A\in\C^{n\times n}$ be an arbitrary Hermitian matrix that is colorable with $c$ colors.  Then, for any real diagonal matrix 
$B\in\C^{n\times n}$, 
\begin{equation}
\lambda_{\max}(B-A) \ge \lambda_{\max} \left( B + \frac{1}{c-1} A\right)
\end{equation}
\end{lemma}
This result implies immediately several known lower bounds on the chromatic number.  It is convenient to formulate the following corollary to obtain derivations of these bounds.
\begin{corollary}\label{cor}
Let $A\in\C^{n\times n}$ be an arbitrary Hermitian matrix that is colorable with $c$ colors.  Then, for any real diagonal matrix 
$B\in\C^{n\times n}$, 
\begin{equation}
\lambda_{\max}(B-A) \ge \lambda_{\max} ( B + A ) - \frac{c-2}{c-1} \lambda_{\max}(A)
\end{equation}
and consequently
\begin{equation}
c \ge 1 + \frac{\lambda_{\max}(A)}{\lambda_{\max}(A) - \lambda_{\max}(B+A) + \lambda_{\max}(B-A)}.
\end{equation}
\end{corollary}
To obtain the statement of the corollary, consider the statement of Lemma~\ref{lem:Niki} and write the matrix on the right hand side as $B+\frac{1}{c-1}A = B+A - \frac{c-2}{c-1} A$.  It is easy to see that $\lambda_{\max}(X-Y)\ge\lambda_{\max}(X)-\lambda_{\max}(Y)$ holds for arbitrary
Hermitian matrices. In particular, this inequality holds for $X=B+A$ and $Y=\frac{c-2}{c-1} A$, which yields the statement of the corollary.

\begin{proof}
Hoffman's bound (\ref{eq:Hoffman}) is equivalent to the statement of Corollary~\ref{cor} when $A$ is the adjacency matrix and $B$ is the zero matrix.  Kolotilina's bounds (\ref{eq:Kolo1}) and (\ref{eq:Kolo2})
are equivalent to the statement of Corollary~\ref{cor} when $A$ is the adjacency matrix and $B=\pm D$, respectively.
Note that Nikiforov's hybrid bound (\ref{eq:Niki}) follows from 
(\ref{eq:Kolo1}) since $\delta_1 \ge 2\mu_1$.
\end{proof}

\section{Proof of Theorem~\ref{thm:genBounds}}
For an arbitrary Hermitian matrix $X\in\C^{n\times n}$, let $\lambda^\downarrow_1(X),\ldots,\lambda^\downarrow_n(X)$ denote its eigenvalues sorted in non-increasing order.

We use the following result, which is well known in majorization theory \cite{bhatia}.
Let $X_1,\ldots,X_d\in\C^{n\times n}$ be arbitrary Hermitian matrices.  Then
\begin{equation}\label{eq:resOne}
\sum_{i=1}^m \lambda_i^\downarrow(X_1) + \ldots + \sum_{i=1}^m \lambda_i^\downarrow(X_d) \ge 
\sum_{i=1}^m \lambda_i^\downarrow\left( X_1 + \ldots + X_d\right).
\end{equation}
Let $S,T\in\C^{n\times n}$ be two arbitrary Hermitian matrices.  Setting $d=2$, $X_1=S-T$ and $X_2=T$, implies the bound
\begin{equation}\label{eq:resTwo}
\sum_{i=1}^m \lambda_i^\downarrow( S - T) \ge \sum_{i=1}^m \lambda_i^\downarrow(S) - \sum_{i=1}^m \lambda_i^\downarrow(T).
\end{equation}

We are now ready to formulate and prove our multi-eigenvalue generalization of Lemma~\ref{lem:Niki} and Corollary~\ref{cor}.

\begin{lemma}\label{lem:genNiki}
Let $A\in\C^{n\times n}$ be an arbitrary Hermitian matrix (with zeros on the diagonal) that can be colored with $c$ colors.  Then, for any real diagonal matrix 
$B\in\C^{n\times n}$,
\begin{equation}
\sum_{i=1}^m \lambda_i^\downarrow (B-A) \ge \sum_{i=1}^m \lambda_i^\downarrow \left( B + \frac{1}{c-1} A\right) 
\end{equation}
for all $m=1\,\ldots,n$.
\end{lemma}

\begin{proof}
In \cite{WE}, the authors proved the following conversion result: there exist $c-1$ diagonal matrices $U_s$ whose diagonal entries are complex roots of unity such that 
\[
\sum_{s=1}^{c-1}  U_s^\dagger ( - A) U_s  = A.
\]
This conversion result immediately implies 
\begin{equation}\label{eq:matEq}
\sum_{s=1}^{c-1}  U_s^\dagger (B - A) U_s  = (c-1) B + A.
\end{equation}
since $U_s B U_s^\dagger = B U_s U_s^\dagger = B$ holds because the diagonal matrices $U_s$ and $B$ commute and $U_s U_s^\dagger = I$ for all $s$.
The statement
\begin{equation}\label{eq:desiredOne}
\sum_{i=1}^m \lambda_i^\downarrow (B-A) \ge \sum_{i=1}^m \lambda_i^\downarrow \left( B + \frac{1}{c-1} A\right) 
\end{equation}
is obtained by applying the result in (\ref{eq:resOne}) with $X_s = U_s^\dagger(B-A) U_s$ for $s=1,\ldots,c-1$ to the left hand side of (\ref{eq:matEq}) and by dividing everything by $(c-1)$.  This uses that conjugation by a unitary matrix does not change the spectrum of a Hermitian matrix, that is, $\lambda_i^\downarrow (X_s)=\lambda_i^\downarrow (B-A)$ for all $i$ and $s$.
\end{proof}

As noted in the introduction, the above result encompasses the special case
\[
\lambda^\downarrow_1( B - A ) \ge 
\lambda^\downarrow_1\left(
 B + \frac{1}{c-1} A
 \right), 
\]
which was proved by Nikiforov in \cite[Theorem 1]{Nikiforov07} using entirely different techniques. 

\begin{corollary}\label{cor:genCor}
We have
\begin{equation}\label{eq:genNikiLHS}
\sum_{i=1}^m \lambda_i^\downarrow \left( B + \frac{1}{c-1} A\right)
\ge
\sum_{i=1}^m \lambda_i^\downarrow (B+A) - \frac{c-2}{c-1} \sum_{i=1}^m \lambda_i^\downarrow (A)
\end{equation}
and consequently 
\begin{equation}\label{eq:genCorSecondIneq}
c \ge 1 + \frac{\sum_{i=1}^m \lambda_i^{\downarrow}(A)}{\sum_{i=1}^m\lambda_i^{\downarrow}(A) - \sum_{i=1}^m \lambda_i^{\downarrow}(B+A) + \sum_{i=1}^m \lambda_i^{\downarrow}(B-A)}.  
\end{equation}
\end{corollary}

\begin{proof}
The first statement 
is obtained by writing $B + \frac{1}{c-1} A=B+A - \frac{c-2}{c-1}A$ and applying the result in (\ref{eq:resTwo}) with $S=B+A$ and $T=\frac{c-2}{c-1}A$ to the left hand side of (\ref{eq:genNikiLHS}).  The second statement follows from the first one by elementary algebra.
\end{proof}

We are now ready to prove the multi-eigenvalue bounds of Theorem~\ref{thm:genBounds}.

\begin{proof}
The generalized Hoffman bound (\ref{eq:genHof}) is equivalent to the statement of Corollary~\ref{cor:genCor} when $A$ is the adjacency matrix and $B$ is the zero matrix.  The generalized Kolotilina bounds (\ref{eq:genKolo1}) and (\ref{eq:genKolo2})
are equivalent to the statement of eq.~(\ref{eq:genCorSecondIneq}) in Corollary~\ref{cor:genCor} when $A$ is the adjacency matrix and $B=\pm D$, respectively.
Note that the generalized Nikiforov bound (\ref{eq:genNiki}) follows from 
(\ref{eq:genKolo1}) since $\delta_i \ge 2\mu_i$ for all $i$.
\end{proof}

Using the conversion result, we can also give an alternative proof of the Lima, Oliveira, Abreu and Nikiforov bound in (\ref{eq:LOAN}). 

\begin{proof}
The identity $D-Q=-A$, and the invariance of the diagonal entries under conjugation by the diagonal unitary matrices $U_s$ imply
\[
A = \sum_{s=1}^{c-1} U_s(-A)U_s^\dagger = \sum_{s=1}^{c-1} U_s(D-Q)U_s^\dagger = (c-1) D - \sum_{s=1}^{c-1} U_s Q U_s^\dagger.
\] 
Define the column vector $v=\frac{1}{\sqrt{n}}(1,1,\ldots,1)^T$. Multiply the left and right most sides of the above matrix equation by $v^\dagger$ from the left and by $v$ from the right to obtain
\[
\frac{2m}{n} = v^\dagger A v = (c-1) \frac{2m}{n} - \sum_{s=1}^{c-1} v^\dagger U_s Q U_s^\dagger v \le
(c-1) \frac{2m}{n} - (c-1) \delta_n.
\] 
This uses that $v^\dagger A v = v^\dagger D v = 2m/n$, which is equal to the sum of all entries of respectively $A$ and $D$ divided by $n$ due to the special form of $v$, and that $w^\dagger U_s Q U_s^\dagger w \ge \lambda_{\min}(Q)=\delta_n$.  This inequality follows from \cite[Problem I.6.15]{bhatia} since $U_s^\dagger w$ is a unit vector, which is not necessarily an eigenvector of $Q$ corresponding to the eigenvalue $\delta_n$.
\end{proof}


\section{Unification of bounds}

Let $G$ be a graph with no isolated vertices. Let $D$ denote the diagonal matrix whose entries $d_1,\ldots,d_n$ are the degrees of the vertices.  Chung \cite{Chung97} defined a normalized adjacency matrix of $G$, $\cA = D^{-1/2} A D^{-1/2}$, and similarly a normalized Laplacian matrix $\cL = I -\cA$ and a normalized signless Laplacian matrix $\cQ = I + \cA$. Let $1 = \mu_1^* \ge \mu_2^* \ge \ldots \ge \mu_n^*$ denote the eigenvalues of $\cA$; $\theta_1^* \ge \theta_2^* \ge \ldots \ge \theta_n^* = 0$ denote the eigenvalues of $\cL$; and  $\delta_1^* \ge \delta_2^* \ge \ldots \ge \delta_n^*$ denote the eigenvalues of $\cQ$. Note that $\theta_i^* = 1 - \mu_{n-i+1}^*$ and $\delta_i^* = 1 + \mu_i^*$, for all $i$.

In Corollary~\ref{cor}, consider the three cases: $B = 0$ and $A=\cA$, $B = I$ and $A = \cL$, and $B = -I$ and $A = \cQ$. These lead to normalized versions of the Hoffman and Kolotilina bounds. However, because of the relationships between the eigenvalues of $\cA$, $\cL$, and $\cQ$, all three normalized bounds are equal to the following inequality: 
\begin{equation}\label{eq:norHoffman}
\chi \ge 1 + \frac{1}{-\mu_n^*}
\end{equation}
Bound (\ref{eq:norHoffman}) therefore unifies the Hoffman and Kolotilina bounds and is equivalent to a special case of \cite[Theorem 6.7]{Chung97}.
Bound (\ref{eq:norHoffman}) is exact for all bipartite graphs and for regular graphs for which bound (\ref{eq:Hoffman}) is exact.   It is also exact for some irregular graphs for which bound (4) is not exact, such as Sierpi\'nski and some Windmill graphs.  

In Corollary~\ref{cor:genCor}, consider again the three cases: $B=0$ and $A=\cA$, $B = I$ and $A = \cL$, and $B = -I$ and $A = \cQ$. These lead to normalized versions of bounds (\ref{eq:genHof}), (\ref{eq:genKolo1}) and (\ref{eq:genKolo2}), all of which are equal to:
\begin{equation}\label{eq:norGenHof}
\chi \ge 1 + \frac{\sum_{i=1}^m \mu_i^*}{\sum_{i=1}^m -\mu^*_{n+i-1}}
\end{equation}
for all $m = 1,\ldots, n$.  
										
These normalized bounds are equal to the equivalent un-normalized bounds for regular graphs. However, bounds (\ref{eq:norHoffman}) and (\ref{eq:norGenHof}) perform better than the un-normalized bounds for many named irregular graphs. 

There are graphs for which each of the bounds discussed in this paper
performs the best. For example the NoPerfectMatching Graph on $16$ vertices,
with $\chi = 4$, has bound (\ref{eq:Hoffman}) equal to $2.5$ but bound (\ref{eq:genKolo1}) with $m = 3$ is the
best, equal to $2.9$. Circulant(16,(1,7,8)), with $\chi = 4$, has bound (\ref{eq:Hoffman})
equal to $2.7$ but bound (\ref{eq:norGenHof}) with $m = 3$ is the best, equal to $2.9$.

\section{Conclusions}
This paper generalises an eigenvalue inequality due to Nikiforov. This enables several lower bounds for the chromatic number to be generalised to encompass all eigenvalues of the adjacency, Laplacian and signless Laplacian matrices. A  bound using the normalized adjacency matrix is also derived, which often performs better than any of the un-normalized bounds.

The proof of Theorem~\ref{thm:genBounds} is straightforward because of the power of combining the conversion result with majorization, and because the proof uses graph matrices rather than the eigenvectors of these matrices.


\section*{Acknowledgements}
We would like to thank Vladimir Nikiforov, in private correspondence, for helpful discussions, leading to the simple derivation of the lower bounds on the chromatic number in Theorem~\ref{thm:bounds}.

P.W. gratefully acknowledges the support from the National Science Foundation CAREER Award CCF-0746600.


\begin{thebibliography}{10}

\bibitem{bhatia}
R.~Bhatia, \emph{Matrix analysis}, Graduate text in mathematics, vol. 169, Springer.

\bibitem{Chung97}
F.~R.~K.~Chung, \emph{Spectral Graph Theory}, CBMS Number 92, 1997.

\bibitem{Hoffman70}
A.~J.~Hoffman, \emph{On eigenvalues and colourings of graphs}, in: Graph Theory and its Applications, Academic Press, New York (1970), pp.~79--91.

\bibitem{Kolotilina}
L.~Yu.~Kolotilina, \emph{Inequalities for the extreme eigenvalues of block-partitioned Hermitian matrices with applications to spectral graph theory}, Journal of Mathematical Sciences, vol. 176, no. 1, July 2011; translation of the paper originally published in Russian in Zapiski Nauchnykh Seminarov POMI, 2010, 382,  82--103.

\bibitem{LOAN}
L.~S.~de Lima, C.~S.~Oliveira, N.~M.~M.~de Abreu, V.~Nikiforov, \emph{The smallest eigenvalue of the signless Laplacian}, Linear Algebra and its Applications, vol. 435, issue 10, (2011), 2570 - 2584.

\bibitem{Nikiforov07}
V. Nikiforov, \emph{Chromatic number and spectral radius}, Linear Algebra Appl. 426, 810-814, 2007.

\bibitem{WE}
P.~Wocjan and C.~Elphick, \emph{New spectral bounds on the chromatic number encompassing all eigenvalues of the adjacency matrix}, Elec. J. Combin. 20(3), (2013), P39. 

\end{thebibliography}
\end{document}